\documentclass{article}
\usepackage[utf8]{inputenc}
\usepackage{authblk}
\usepackage{amsthm}
\usepackage{amssymb}
\usepackage{amsmath}
\usepackage{biblatex}
\usepackage{amsfonts}
\usepackage{mathtools}
\usepackage{tikz}
\usepackage{authblk}
\usepackage{mathrsfs}

\usetikzlibrary{shadows}
\usetikzlibrary{arrows}
\usetikzlibrary{decorations.markings}
\theoremstyle{plain}
\newtheorem{theorem}{Theorem}[section]
\newtheorem{lemma}[theorem]{Lemma}

\theoremstyle{definition}
\newtheorem{cor}{Corollary}[theorem]

\title{Counting Isomorphism Classes of Spanning Trees of Complete Bipartite Graphs}

\author[a]{Peter Johnson\\Auburn University\\ johnspd@auburn.edu \and Shayne Nochumson\\Auburn University\\ snk0018@auburn.edu}

\usepackage[colorlinks=true, citecolor=black, linkcolor=black]{hyperref}

\begin{document}

\maketitle


\noindent

\begin{abstract}
    Spanning trees of complete bipartite graphs exhibit a rich interaction between degree sequences and graph structure.  In this paper, we obtain lower bounds on the number of isomorphism classes of spanning trees in $K_{a,b}, 2 \leq a \leq b$ in terms of $P_a(a+b-1)$ and $P_b(a+b-1)$ where $P_k(m)$ is the number of integer partitions of $m$ of length $k$. Furthermore, we obtain an upper bound in terms of $a$ and $b$.
\end{abstract}

\section{Introduction}
For a given graph $G$, we denote by $\tau(G)$ the number of labeled spanning trees of $G$.  $\tau(G)$ has been widely studied with classical results such as Kirchoff's Matrix-Tree Theorem \cite{KirchhoffUeberDA} from 1847 and Cayley's formula \cite{borchardt1861} proven by Borchardt in 1861. In 1962, Scoins proved that $\tau(K_{a,b})=a^{b-1}\cdot b^{a-1}$ \cite{Scoins_1962}.  The following year, Glicksman gave an alternate proof for Scoins' formula \cite{Glicksman}.  We consider a question related to that of Scoins.  How many isomorphism classes of spanning trees are there of $K_{a,b}$?  We denote this value by $I_{a,b}$ and prove a lower bound by considering integer partitions of $a+b-1$. We then prove a upper bound on $I_{a,b}$ which can be calculated using only $a$ and $b$.

\section{Main Results}
\begin{lemma} Let $2 \leq a \leq b$ and $a+b=n$. For any pair of integer partitions $s_1 \geq ... \geq s_a >0$ and $t_1 \geq ... \geq t_b >0$ of $a+b-1$, there is a connected bipartite graph $Q = Q((s_i),(t_j))$ with bipartitions $A,B$ where $|A|=a, |B|=b$ and where the $s_i$ are the degrees in $Q$ of vertices in $A$ and the $t_j$ are the degrees in $Q$ of the vertices in $B$.
\end{lemma}

\begin{proof} 

 We will proceed by induction on $a+b=n$.  The smallest allowable $n$ is $n=2+2=4$.  The only partition of length $2$ of $n-1=3$ is $3=2+1$.

\begin{figure}[h]
    \centering
    
\begin{tikzpicture}[
           vs/.style={circle,draw=black,text=black},
           es/.style={thick, black},edge_dash/.style={ultra thick, black,dashed},scale=.5]
                      
            \node[vs,scale=.8,fill=white,  label={left:$1$}] (1) at (0,0){};
            \node[vs,scale=.8,fill=white, label={right:$1$}] (2) at (2,0){};
            \node[vs,scale=.8,fill=white, label={right:$2$}] (3) at (2,2){};
            \node[vs,scale=.8,fill=white, label={left:$2$}] (4) at (0,2){}; 
             
            \draw[es] (1) -- (3);
            \draw[es] (2) --(4);
            \draw[es](3)-- (4);
            
        \end{tikzpicture}   
    
    \caption{$Q((2,1),(2,1))$}
    \label{}
\end{figure}
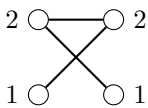

Suppose that $n=a+b>4$; and we have integers $s_1 \geq ... \geq s_a >0$ and $t_1 \geq ... \geq t_b >0$ where $\displaystyle \sum_{i=1}^{a} s_i = \sum_{j=1}^{b} t_j = a+b-1$. If $t_b>1$ then $t_1+...+t_b \geq 2b > a+b-1$.  Therefore, $t_b=1$.  On the other hand, $a< a+b-1 = s_1+...+s_a \leq as_1$, so $s_1 \geq 2$. 

\begin{center}
    Let $s_i' = \begin{cases} 
  s_i -1, & \text{if } i=1 \\
  s_i, & \text{otherwise} 
\end{cases}$
\end{center}

Then $s_i',...,s_a'$ are positive integers summing to $a+(b-1)-1=t_1+...+t_{b-1}$.  By the induction hypotheses, there is a connected bipartite graph $Q'((s_1',...,s_a'),(t_1,...,t_{b-1}))$ with bipartition $A,B'$ where $|A|=a, |B'| = b-1$ and where the $s_i'$ are the $Q'$ degrees of the vertices in $A$ and $t_1,...,t_{b-1}$ are the degrees of $B'$.  Let $u \in A$ be a vertex with degree $s_1 -1$ in $Q'$, let $v \not\in A \cup B'$, and let $B= B' \cup \{v\}.$  Let $Q= Q' \cup uv$.  Then $Q = Q((s_i),(t_j))$ is a connected bipartite graph with bipartition $A,B$ with the vertices in $A$ having degrees $s_1,...,s_a$ and the vertices in $B$ having degrees $t_1,...,t_b$. 

\end{proof}

\begin{lemma} A tree is a spanning tree of $K_{a,b}$ if and only if it is isomorphic to some such $Q$.

\end{lemma}

\begin{proof}
Suppose $T$ is a spanning tree of $K_{a,b}.$  Since $T$ spans $K_{a,b}$, $|V(T)|=a+b$ and $|E(T)|=a+b-1$. Let $A,B$ denote the partite sets of $K_{a,b}$ where $|A|=a$ and $|B|=b$. As $K_{a,b}$ is bipartite, $uv \in E(T)$ if and only if $u$ and $v$ lie in different partite sets. Then the degrees of vertices in $A$ must be an integer partition of $a+b-1$. Similarly for the degrees of $B$.  Thus, $T$ is isomorphic to some $Q((s_i),(t_j))$ where $s_1 \geq ... \geq s_a >0$ and $t_1 \geq ... \geq t_b>0$ are integer partitions of $a+b-1$. \\

Next, consider some $Q((s_i),(t_j))$ for some pair of integer partitions of $a+b-1$, $s_1 \geq ... \geq s_a>0$ and $t_1 \geq ... \geq t_b$.  By Lemma 2.1, $Q((s_i),(t_j))$ is a connected bipartite graph. Since $|E(Q((s_i),(t_j)))|=a+b-1$, $Q((s_i),(t_j))$ must be a tree.  As $Q((s_i),(t_j))$ is a tree with bipartite sets of sizes $a$ and $b$, $Q((s_i),(t_j))$ must be a spanning tree of $K_{a,b}$. 

\end{proof}

 Let $I_{a,b}$ denote the number of isomorphism classes of spanning trees of $K_{a,b}$, and let $P_k(m)$ denote the number of integer partitions of $m$ of length $k$.

 \begin{theorem} If $b>a\geq2$, $I_{a,b} \geq P_a(a+b-1) \cdot P_b(a+b-1)$.
 \end{theorem}

\begin{proof}
 Consider two distinct pairs of integer partitions $((s_1, ... ,s_a),(t_1,...t_b))$ and \\ $((s_1', ... ,s_a'),(t_1',...t_b'))$.  By Lemma 2.1, there is a connected bipartite graph $Q$ $(Q'$    resp.) with bipartitions $A,B$ $(A',B'$ resp.) of sizes $a$ and $b$. By Lemma 2.2, $Q$ and $Q'$ are each a spanning tree of $K_{a,b}$. Since $((s_i),(t_j))$ and $((s_i'),(t_j'))$ differ, the degree sequences of $Q$ and $Q'$ differ, so $Q$ and $Q'$ cannot be isomorphic trees because $b>a$. Then, any ordered pair of integer partitions generates at least one isomorphism class of spanning trees of $K_{a,b}$, and spanning trees corresponding to different pairs are not isomorphic. Thus, $I_{a,b} \geq P_a(a+b-1) \cdot P_b(a+b-1)$. 

 \end{proof}

\begin{cor} If $2 \leq a \leq b$, then $a^{b-1}\cdot b^{a-1} \geq P_a(a+b-1)\cdot P_b(a+b-1)$.
\end{cor}

\noindent Corollary 2.3.1 is a direct consequence of Scoins' formula and Theorem 2.3.

\begin{theorem}If $b=a\geq2$, then $I_{a,a} \geq \dfrac{r(r+1)}{2}$ where $r=P_a(2a-1)$.
\end{theorem}

\begin{proof}

Let $K_{a,b}$ have bipartitions $A,B$ such that $|A|=|B|=a$, let $P=\{p_1,...,p_r\}$ denote the set of integer partitions of $2a-1$ of length $a$, and let $P'=\{p_1',...,p_r'\}$ where $p_i'=p_i$.  For any pair $(s,t)$ where $s \in P$ and $t \in P'$, there exists at least one spanning tree $Q$ of $K_{a,a}$ by Lemmas 2.1 and 2.2. Consider some $(s,t)=(p_k,p_l')$ and $(s',t')=(p_l,p_k')$.  As $|A|=|B|=a$ and $K_{a,b}$ is unlabeled, any spanning tree $Q$ generated by $(s,t)$ must be isomorphic to some spanning tree $Q'$ generated by $(s',t')$.  Then, the number of isomorphism classes of spanning trees of $K_{a,b}$ where $a=b$ must be bounded by the number of unordered pairs $s'',t''$ where $s'' \in P$ and $t'' \in P'$.  There are $r$ pairs such that $s''=t''$ and $\displaystyle { r \choose 2 }$ pairs where $s''\not=t''$.  Thus, there are $r+\dfrac{r(r-1)}{2}=\dfrac{2r+r^2-r}{2}=\dfrac{r(r+1)}{2}$ unique unordered pairs of integer partitions of $2a-1$ where each is of length $a$.  Since each pair must generate at least one isomorphism class, $I_{a,a} \geq \dfrac{r(r+1)}{2}$. 

\end{proof}

\begin{cor}
    If $a \geq 2$, then $a^{2a-2} \geq \dfrac{P_a(2a-1)\cdot (P_a(2a-1)+1)}{2}$.
\end{cor}

\begin{lemma}
    If $2 \leq a < b$, $I_{a,b} \leq I_{a,a} \cdot a^{b-a}$.
\end{lemma}

\begin{proof}
    We proceed by induction on $n=a+b$.  The smallest allowable $n$ is $n=5=2+3$.  In this case, $G=K_{2,3}$.  As shown below, $I_{2,2}=1$ and $I_{2,3}=2$. Thus, $2= I_{2,3} \leq I_{2,2} \cdot 2^{3-2} = 1 \cdot 2 = 2$.

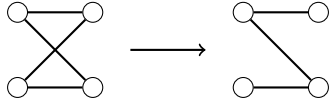
\begin{figure}[h]
    \centering
    
\begin{tikzpicture}[
           vs/.style={circle,draw=black,text=black},
           es/.style={thick, black},edge_dash/.style={ultra thick, black,dashed},scale=.5]
                      
            \node[vs,scale=.8,fill=white] (1) at (0,0){};
            \node[vs,scale=.8,fill=white] (2) at (0,2){};
            \node[vs,scale=.8,fill=white] (3) at (2,0){};
            \node[vs,scale=.8,fill=white] (4) at (2,2){};
            \node[vs,scale=.8,fill=white] (5) at (6,0){}; 
            \node[vs,scale=.8,fill=white] (6) at (6,2){}; 
            \node[vs,scale=.8,fill=white] (7) at (8,0){}; 
            \node[vs,scale=.8,fill=white] (8) at (8,2){}; 
             
            \draw[es] (1) -- (3);
            \draw[es] (2) --(4);
            \draw[es](2)-- (3);
            \draw[es](1)-- (4);
            \draw[es](5)-- (7);
            \draw[es](6)-- (8);
            \draw[es](6)-- (7);
            \draw[es,->]        (3,1)   -- (5,1);
        \end{tikzpicture}   
    
    \caption{$K_{2,2}$ and its only isomorphism class of spanning trees.}
    \label{}
\end{figure}

\begin{figure}[h]
    \centering
    
\begin{tikzpicture}[
           vs/.style={circle,draw=black,text=black},
           es/.style={thick, black},edge_dash/.style={ultra thick, black,dashed},scale=.5]
                      
            \node[vs,scale=.8,fill=white] (1) at (0,2){};
            \node[vs,scale=.8,fill=white] (2) at (0,4){};
             \node[vs,scale=.8,fill=white] (3) at (2,1){};
             \node[vs,scale=.8,fill=white] (4) at (2,3){};
             \node[vs,scale=.8,fill=white] (5) at (2,5){};

             \node[vs,scale=.8,fill=white] (6) at (6,10){};
             \node[vs,scale=.8,fill=white] (7) at (6,12){};
             \node[vs,scale=.8,fill=white] (8) at (8,9){};
             \node[vs,scale=.8,fill=white] (9) at (8,11){};
             \node[vs,scale=.8,fill=white] (10) at (8,13){};

             \node[vs,scale=.8,fill=white] (11) at (12,2){};
             \node[vs,scale=.8,fill=white] (12) at (12,4){};
             \node[vs,scale=.8,fill=white] (13) at (14,1){};
             \node[vs,scale=.8,fill=white] (14) at (14,3){};
             \node[vs,scale=.8,fill=white] (15) at (14,5){};
             
            \draw[es] (1) -- (3);
            \draw[es] (1) -- (4);
            \draw[es] (2) -- (4);
            \draw[es] (2) -- (5);

            \draw[es] (6) -- (8);
            \draw[es] (6) -- (9);
            \draw[es] (6) -- (10);
            \draw[es] (7) -- (8);
            \draw[es] (7) -- (9);
            \draw[es] (7) -- (10);

            \draw[es] (11) -- (13);
            \draw[es] (11) -- (14);
            \draw[es] (11) -- (15);
            \draw[es] (12) -- (15);
            
            \draw[es,->] (5,8)   -- (3,6);
            \draw[es,->] (9,8)   -- (11,6);
        \end{tikzpicture}   
    
    \caption{$K_{2,3}$ and its $2$ isomorphism classes of spanning trees.}
    \label{}
\end{figure}
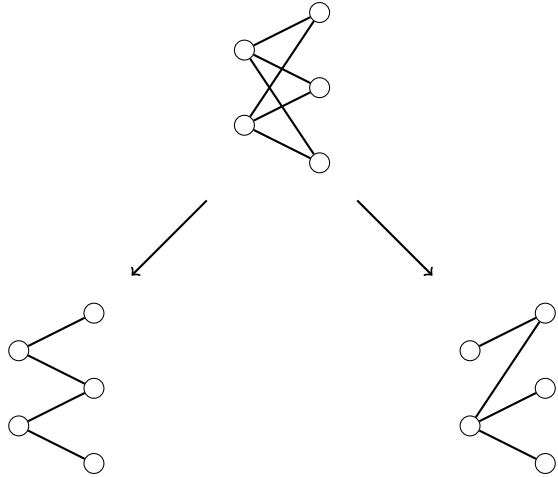

Suppose that $n=a+b>5$ and that $G$ is the complete bipartite graph of order $n$ with bipartitions $A$ and $B$ where $|A|=a$ and $|B|=b$. Let $G'$ be the complete bipartite graph with bipartition $A$, $B'$ where $|A|=a$ and $|B'|=b-1$. By the induction hypothesis, $I_{a,b-1} \leq I_{a,a}\cdot a^{b-1-a}$.  Suppose that $T'$ is a spanning tree of $G'$.  Let $v \not\in A \cup B'$ and $B=B' \cup \{v\}$.  Since $N(v)=A$ and $T'$ spans $G'$, for any $u \in A$, $T=T' \cup uv$ is a spanning tree of $G$.  Then, $I_{a,b} \leq I_{a,a} \cdot a^{b-1-a} \cdot a = I_{a,a} \cdot a^{b-a}$ as $|N(v)|=|A|=a$.
    
\end{proof}

\begin{theorem}
    If $2 \leq a < b$, $I_{a,b} \leq a^{a+b-2}$. 
\end{theorem}

\begin{proof}
    By Lemma 2.5, $I_{a,b} \leq I_{a,a} \cdot a^{b-a}$. Since $I_{a,a} \leq a^{2a-2}$ by Scoins' formula, $I_{a,b} \leq I_{a,a} \cdot a^{b-a} \leq a^{2a-2}\cdot a^{b-a} = a^{a+b-2}$.
\end{proof}

Note that Theorem 2.6 gives a nontrivial bound when $2\leq a <b$.  Scoins' formula gives an obvious upper bound of $a^{b-1} \cdot b^{a-1}$ as it enumerates the number of labeled spanning trees of $K_{a,b}$.  However, when $2 \leq a < b$, $a^{a+b-2}=a^{a-1}\cdot a^{b-1} < a^{b-1} \cdot b^{a-1}$.

\printbibliography

@article{Scoins_1962, title={The number of trees with nodes of alternate parity}, volume={58}, number={1}, journal={Mathematical Proceedings of the Cambridge Philosophical Society}, author={Scoins, H. I.}, year={1962}, pages={12–16}}

@article{Glicksman, title={On the representation and enumeration of trees}, volume={59}, number={3}, journal={Mathematical Proceedings of the Cambridge Philosophical Society}, author={Glicksman, Stephen}, year={1963}, pages={509–517}}

@article{KirchhoffUeberDA,
  title={Ueber die Aufl{\"o}sung der Gleichungen, auf welche man bei der Untersuchung der linearen Vertheilung galvanischer Str{\"o}me gef{\"u}hrt wird},
  author={Gustav R. Kirchhoff},
  journal={Annalen der Physik}, year={1847},
  volume={148},
  pages={497-508},
}

@article{borchardt1861,
  author = {Borchardt, Carl Wilhelm},
  title = {Über eine Interpolationsformel für eine Art symmetrischer Functionen und über deren Anwendung},
  journal = {Monatsberichte der Königlich Preußischen Akademie der Wissenschaften zu Berlin},
  year = {1861},
  pages = {257-266},
}

\end{document}